\documentclass{amsart}

\usepackage[margin=1in]{geometry}
\usepackage[table]{xcolor}
\usepackage{amssymb}
\usepackage[hidelinks]{hyperref}
\usepackage[noabbrev,capitalize]{cleveref}
\usepackage{microtype}
\usepackage{tikz}
\usepackage{tikz-qtree}
\usepackage{tkz-graph}
\usepackage{tkz-berge}
\usepackage{subcaption}
\usepackage{pgfplots}
\usepgfplotslibrary{fillbetween}
\usetikzlibrary{calc}
\pgfdeclarelayer{bg}
\pgfsetlayers{bg,main}

\usepackage[style=authoryear,
giveninits=true,
isbn=false,
doi=false,
url=false]{biblatex}

\makeatletter

\newrobustcmd*{\parentexttrack}[1]{%
  \begingroup
  \blx@blxinit
  \blx@setsfcodes
  \blx@bibopenparen#1\blx@bibcloseparen
  \endgroup}

\AtEveryCite{%
  }

\makeatother

\renewbibmacro{in:}{}
    
\newtheorem{thm}{Theorem}

\newtheorem{lem}[thm]{Lemma}

\Crefname{lem}{Lemma}{Lemmas}

\theoremstyle{definition}

\title{The optimal packing of eight points in the real projective plane}
\author{Dustin G. Mixon \and Hans Parshall}

\begin{document}

\begin{abstract}

How can we arrange $n$ lines through the origin in three-dimensional Euclidean space in a way that maximizes the minimum interior angle between pairs of lines?  Conway, Hardin and Sloane (1996) produced line packings for $n \leq 55$ that they conjectured to be within numerical precision of optimal in this sense, but until now only the cases $n \leq 7$ have been solved. In this paper, we resolve the case $n = 8$.  Drawing inspiration from recent work on the Tammes problem, we enumerate contact graph candidates for an optimal configuration and eliminate those that violate various combinatorial and geometric necessary conditions.  The contact graph of the putatively optimal numerical packing of Conway, Hardin and Sloane is the only graph that survives, and we recover from this graph an exact expression for the minimum distance of eight optimally packed points in the real projective plane.

\end{abstract}

\maketitle

\section{Introduction}

Consider the fundamental \emph{line packing problem} of packing $n$ points in $\mathbf{RP}^{d-1}$ or $\mathbf{CP}^{d-1}$ so that the minimum distance is maximized.
We refer to such $n$-point sets as \emph{projective $n$-packings}, and we take the distance between two points in projective space to be the interior angle of the corresponding lines in affine space.
The line packing problem has received considerable attention over the last century.
The case of $\mathbf{CP}^1$ is equivalent to packing points in $S^2$, namely, the Tammes problem~\parencite{tammes30}.
The general line packing problem was originally formulated in~\parencite{fejes65} and subsequently studied in~\parencite{welch74,delsarte75,levenshtein82}.
Several putatively optimal packings in $\mathbf{RP}^{d-1}$ were later provided in~\parencite{conway96}, each conjectured to be within numerical precision of an optimal packing.
Most of these conjectures remain open, even despite a surge of progress over the last decade (see~\parencite{fickus18} and references therein) that has been motivated in part by emerging applications in compressed sensing~\parencite{bandeira13}, digital fingerprinting~\parencite{mixon13}, quantum state tomography~\parencite{renes04}, and multiple description coding~\parencite{strohmer03}.

In the present paper, we consider the special case of projective $n$-packings in $\mathbf{RP}^2$, that is, sets of $n$ lines through the origin in $\mathbf{R}^3$.  It is convenient to represent a projective $n$-packing by a set of $n$ unit vectors $\Phi \subseteq S^2$ spanning the corresponding lines, where two such sets represent the same packing if and only if one can be obtained from the other by negating some of the vectors.  We denote the resulting equivalence class by $[\Phi]$.  The \emph{coherence} of the projective $n$-packing $\Phi$ is defined as
\[
	\mu(\Phi) := \max_{\substack{x,y\in\Phi\\x\neq y}} |\langle x,y \rangle|,
\]
which satisfies $\mu(\Phi')=\mu(\Phi)$ for every $\Phi'\in[\Phi]$.  Setting
\[
	\mu_n := \inf\{\mu(\Phi) : \Phi \subseteq S^2, |\Phi| = n\},
\]
we say that a projective $n$-packing $\Phi$ is \emph{optimal} if $\mu(\Phi) = \mu_n$.  The existence of optimal projective $n$-packings is guaranteed for every $n$ by a compactness argument, but they have only been classified for $n \leq 7$~\parencite{fejes65,conway96,cohn12}; see~\cref{OptPackings} for an illustration.  In this paper, we compute $\mu_8$ and identify an optimal projective $8$-packing, which is unique up to isometry.

\begin{figure}[t]
\begin{center}
\includegraphics[width=0.32\textwidth]{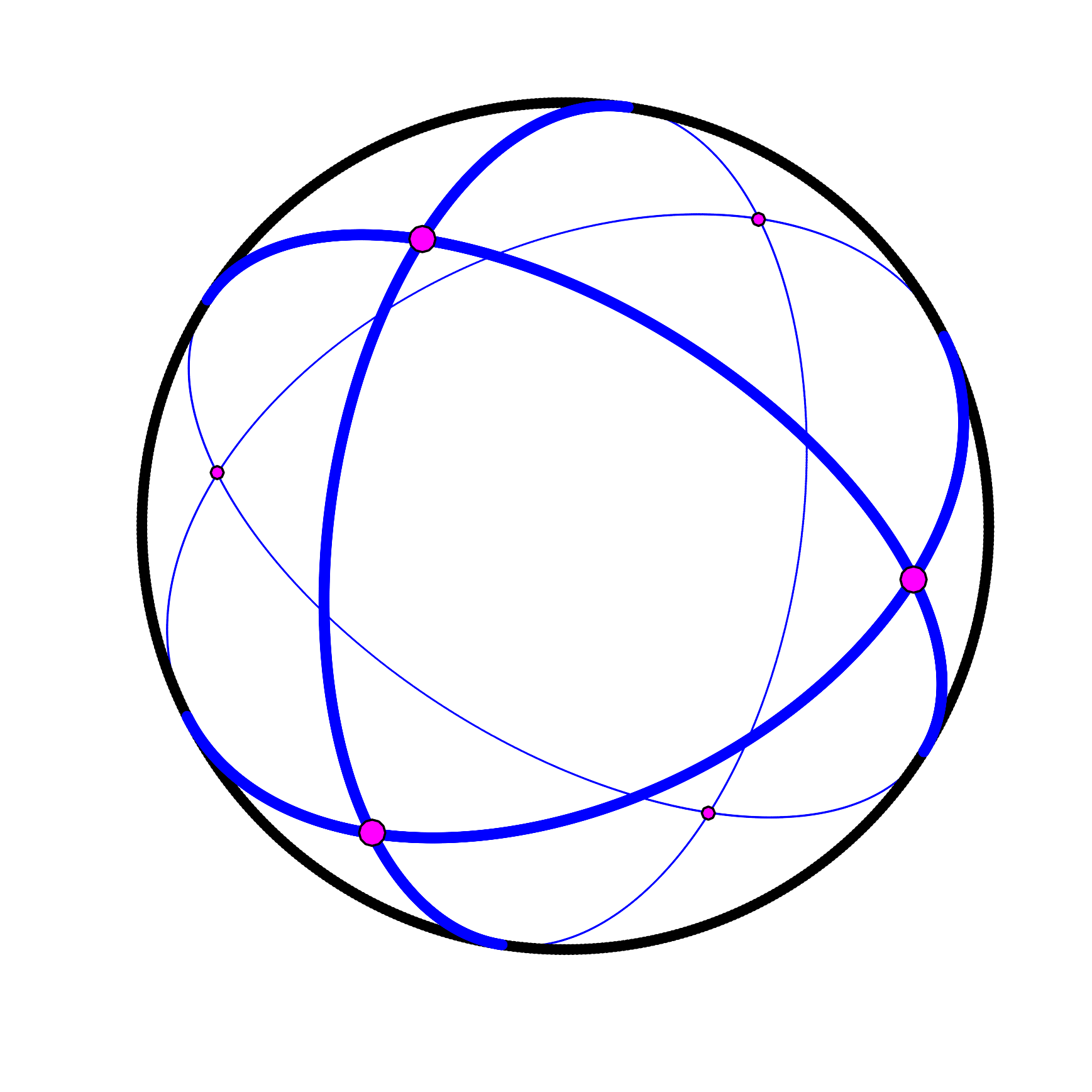}
\includegraphics[width=0.32\textwidth]{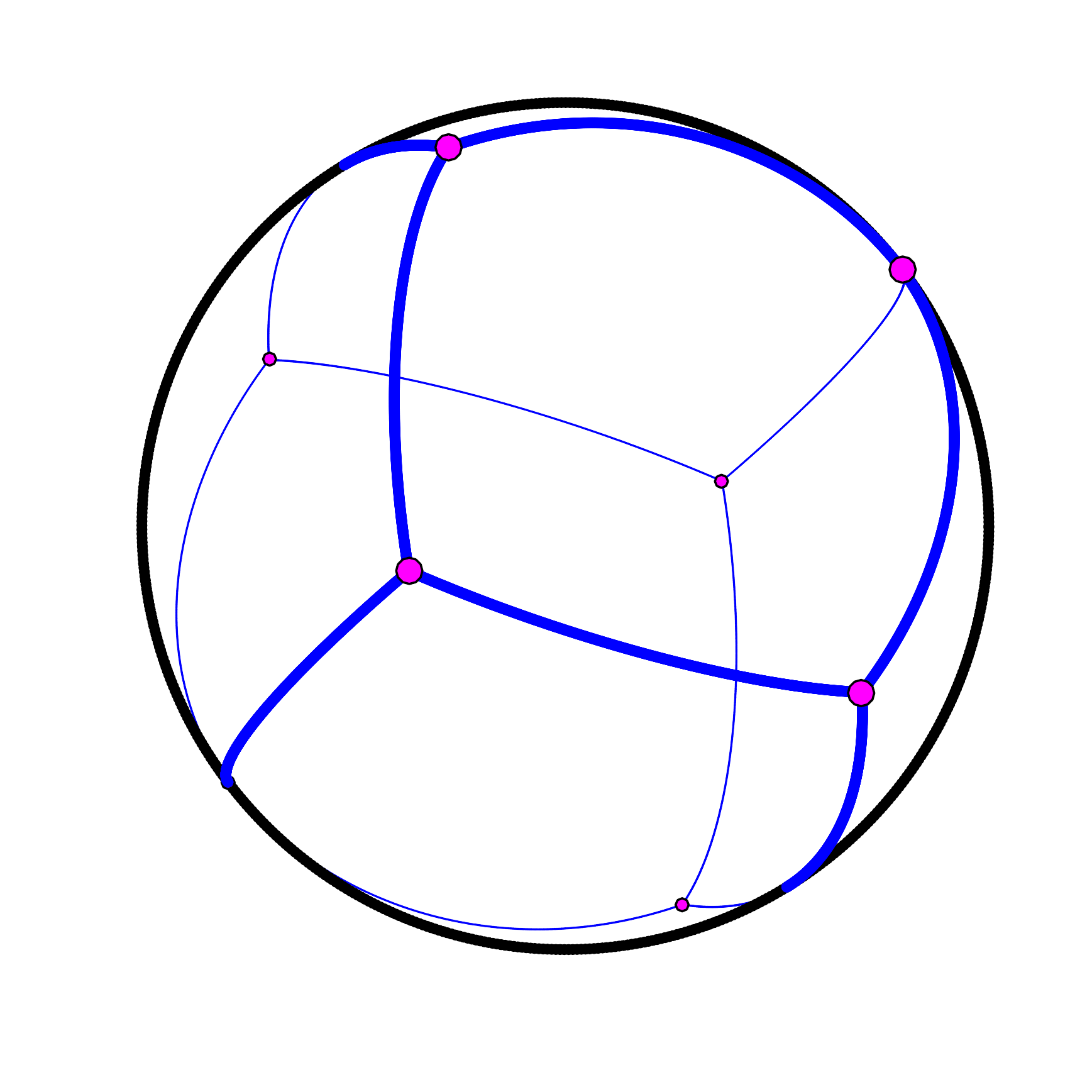}
\includegraphics[width=0.32\textwidth]{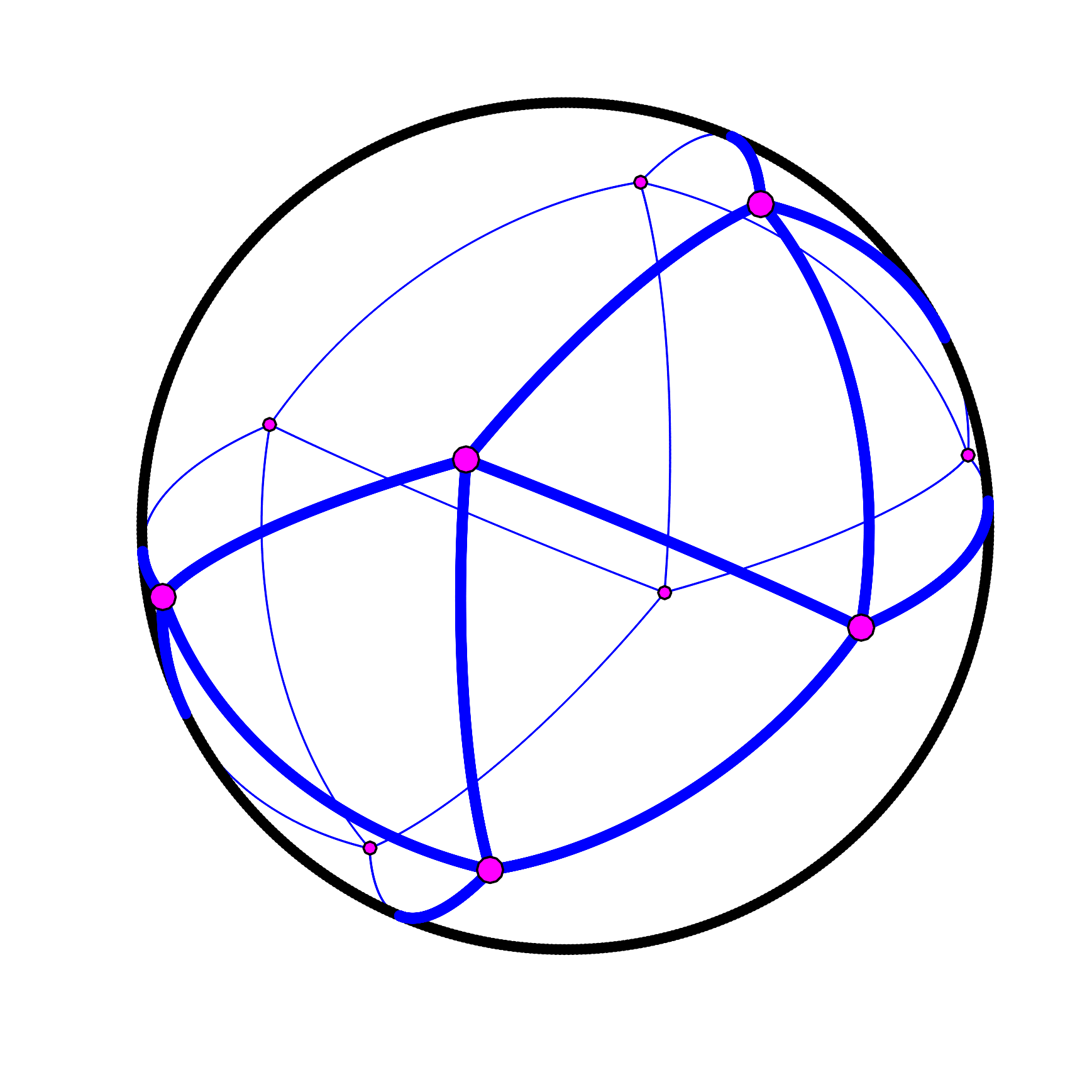}
\includegraphics[width=0.32\textwidth]{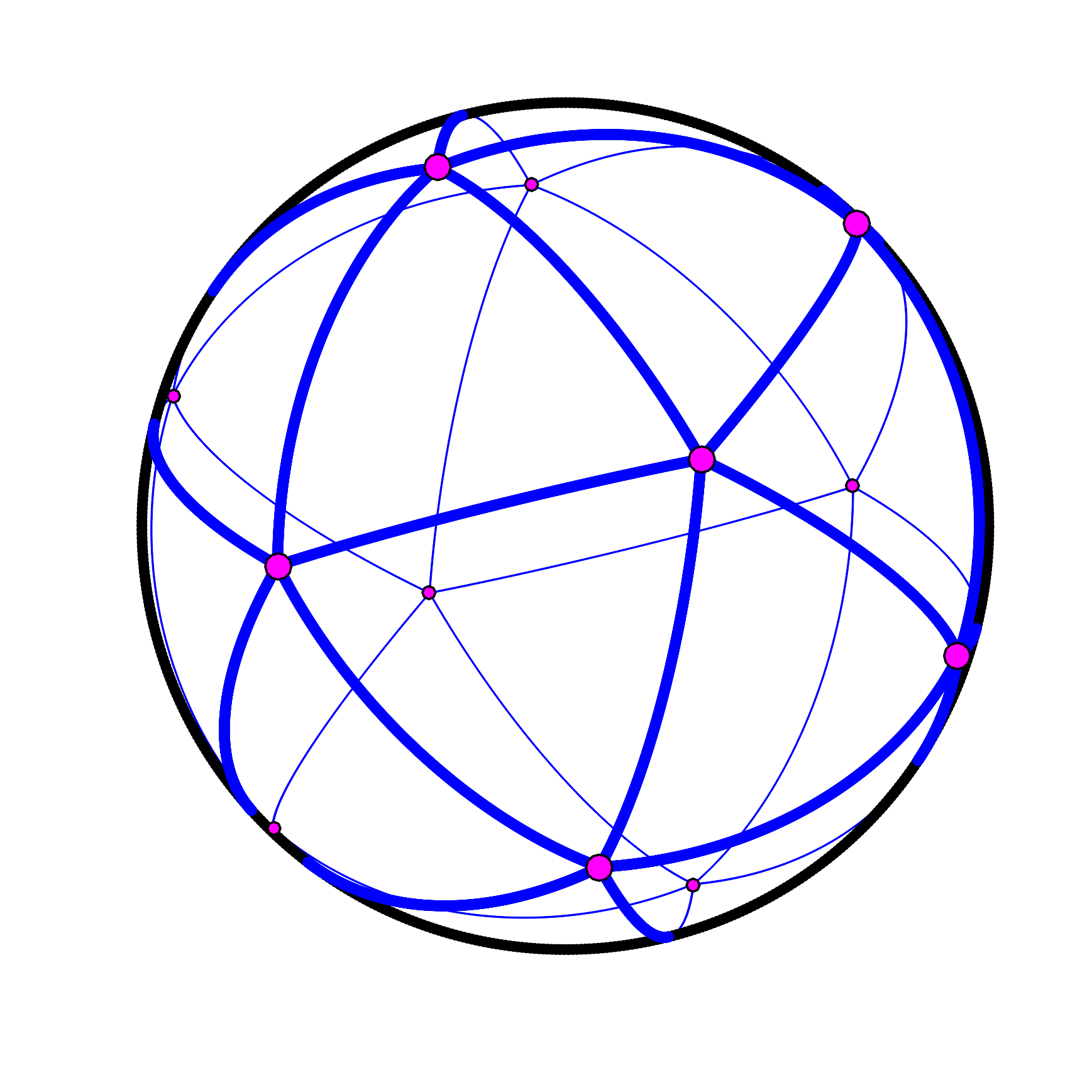}
\includegraphics[width=0.32\textwidth]{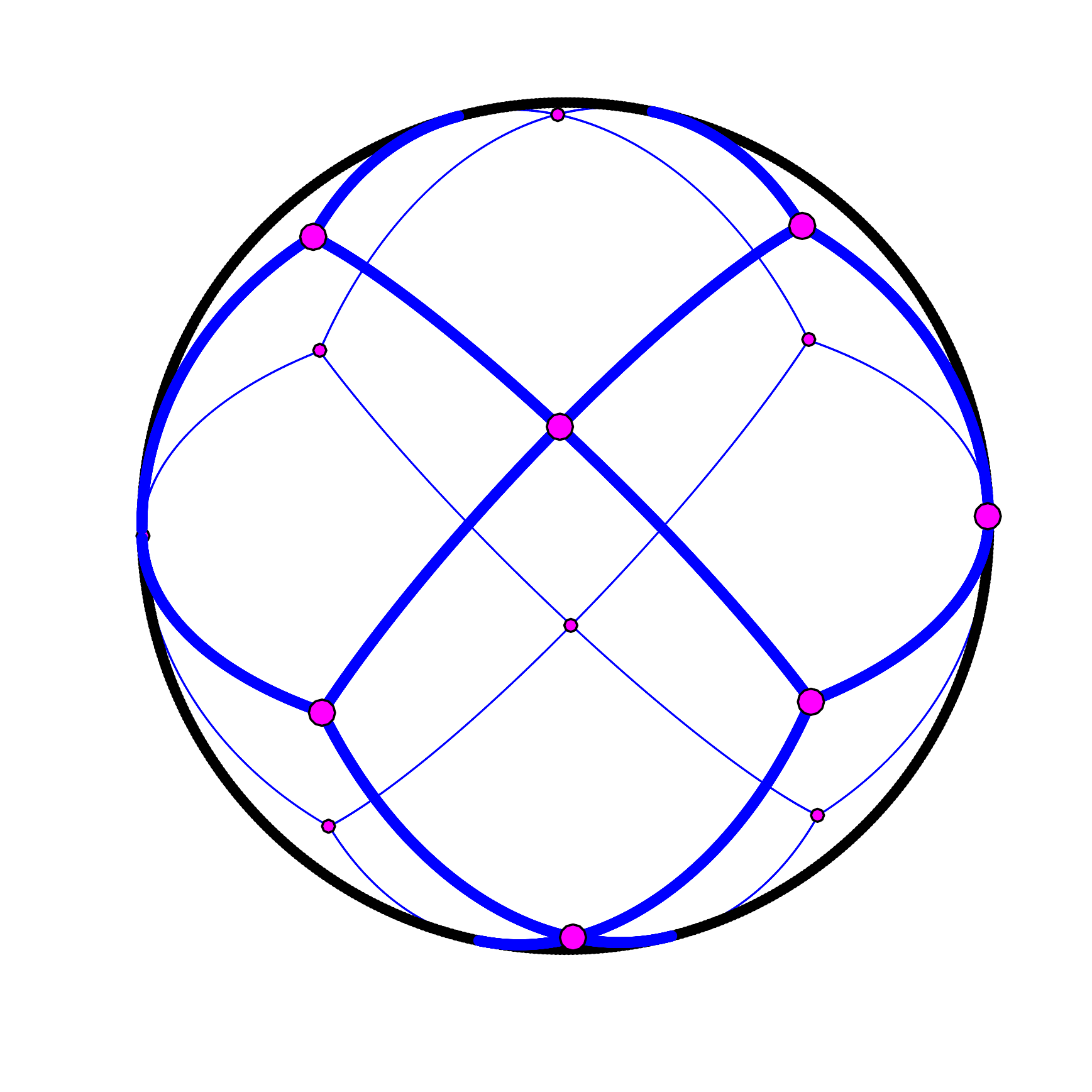}
\includegraphics[width=0.32\textwidth]{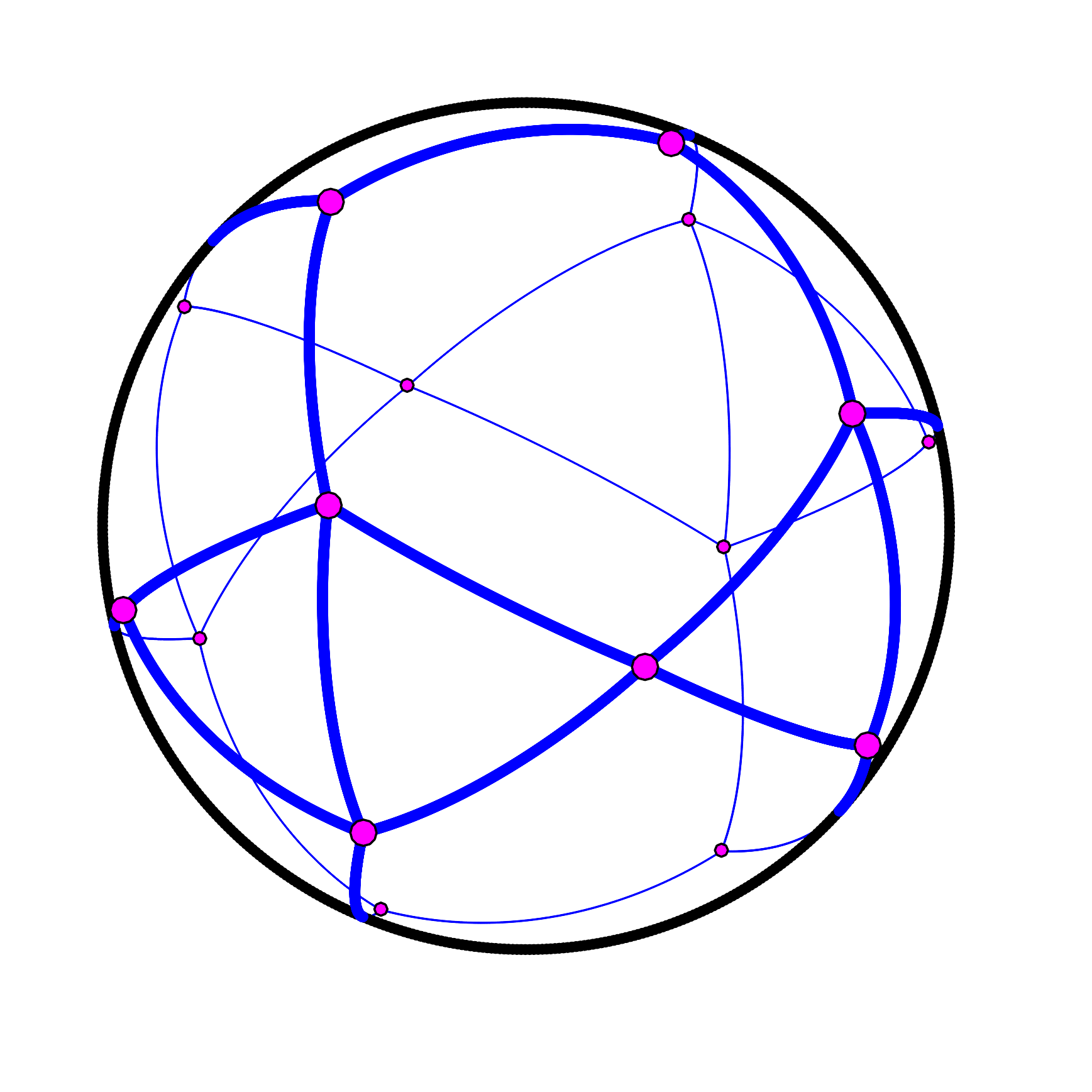}\end{center}
\caption{
Optimal antipodal packings of $2n$ points in $S^2$ for $n\in\{3,\ldots,8\}$, along with the geodesic edges that appear in the corresponding contact graphs.
For the packings with $n\leq 6$, unique optimality was established in~\parencite{fejes65}.
Optimality in the case $n=7$ was proven in~\parencite{conway96}, while the uniqueness of this optimal packing was established in~\parencite{cohn12}.
Our main contribution is a computer-assisted proof of a conjecture in~\parencite{conway96} that the above packing for $n=8$ is optimal (and uniquely so).
\label{OptPackings}}
\end{figure}

\begin{thm}
The projective $8$-packing with minimum coherence $\mu_8$ is unique up to isometry.  Furthermore, $\mu_8$ is the third-largest root of
\[
	1 + 5x - 8x^2 - 80x^3 - 78x^4 + 146x^5 - 80x^6 - 584x^7 + 677x^8 + 1537x^9,
\]
which is given numerically by $\mu_8 \approx 0.6475889787$.
\end{thm}

This improves upon the prior state of the art, which is summarized by the following lemma.
The lower bound follows from the Levenshtein bound~\parencite{levenshtein82}, while the upper bound follows from an explicit numerical packing given in~\parencite{conway96}, which was recently given exact coordinates in~\parencite{mixon19}.

\begin{lem}\label{mu8Lemma}
The coherence $\mu_8$ of an optimal projective 8-packing satisfies 
\[
0.6 \leq \mu_8 \leq 0.647588979.
\]
\end{lem}

A set $X \subseteq S^2$ is said to be \emph{antipodal} if $X = -X$, and we say $X$ is an \emph{antipodal $2n$-packing} if $X$ is antipodal with $|X|=2n$.
We may identify any projective $n$-packing $\Phi$ with an antipodal $2n$-packing, namely
\[
	X(\Phi) := \{x \in S^2 : x \in \Phi \text{ or } {-x} \in \Phi\}.
\]
For $x,y \in S^2$, we define their \emph{geodesic distance} to be
\[
	d(x,y) := \arccos(\langle x,y \rangle),
\]
and we define the \emph{minimum distance} of any finite $X \subseteq S^2$ to be
\[
	\psi(X) := \min_{\substack{x,y \in X \\ x \neq y}} d(x,y).
\]
Of course, minimizing $\mu(\Phi)$ over projective $n$-packings $\Phi$ is equivalent to maximizing $\psi(X)$ over antipodal $2n$-packings $X$.  Both formulations play a role in our proof of the main result.

The second formulation bears resemblance to the \emph{Tammes problem}, originating in \parencite{tammes30}, of maximizing $\psi(X)$ over arbitrary packings of $n$ points $X \subseteq S^2$.  The earliest solutions to the Tammes problem, for $n = 3, 4, 6, 12$, are given by regular triangulations of the sphere that achieve equality in a general inequality on the optimal minimum distance~\parencite{fejes49}.  For sufficiently large values of $n$, this inequality was sharpened in~\parencite{robinson61}, leading to the solution for $n = 24$.  Further progress has been made by studying the \emph{contact graph} of $X \subseteq S^2$, denoted $G(X)$, which is drawn on the sphere with vertices $X$ and with a geodesic edge between $x,y \in X$ exactly when $d(x,y) = \psi(X)$.  Classifying contact graphs led to solutions to the Tammes problem for $n = 5, 7, 8, 9$~\parencite{schutte51} and $n = 10, 11$~\parencite{danzer86}.  More recently, computer enumeration and optimization of over a billion contact graph candidates led to the solutions for $n = 13, 14$~\parencite{musin12,musin15}.  The authors write in related work~\parencite{musinSteklov15} that ``the direct approach to the solution of the Tammes problem based on computer enumeration of irreducible contact graphs \ldots\ has actually been exhausted.''  All other cases of the Tammes problem remain open.  For instance, the $n=20$ case is open, though it has long been known that the vertices of the dodecahedron are suboptimal~\parencite{vdw52}.

To prove our main result, we adapt the approach of contact graph enumeration and elimination.  To make the enumeration feasible, we leverage the antipodal constraint to significantly reduce the number of graphs that we need to generate and consider.  
To this end, Section~\ref{contactSection} identifies a set $\Gamma_1$ of $547$ graphs with the property that every maximal connected set of optimal antipodal $16$-packings has a representative packing with contact graph in $\Gamma_1$.
In Section~\ref{sec.combinatorial}, we consider all planar embeddings of members of $\Gamma_1$ modulo homeomorphism and reflection.
These equivalence classes can be represented as combinatorial embeddings, which we identify using the SPQR-tree data structure developed in~\parencite{dibattista89}.
After removing embeddings that violate various necessary conditions on the contact graph's facial structure, we arrive at a set $\mathcal{E}_1$ of $217$ combinatorial embeddings of contact graph candidates.
Section~\ref{geometricSection} then imposes geometric constraints in order to eliminate all but four of these candidates, three of which are subgraphs of the fourth.
We then demonstrate in Section~\ref{sec.opt and uniqueness} that the three subgraphs violate the Karush--Kuhn--Tucker optimality conditions, thereby isolating a single contact graph candidate $G$. For the sake of reproducibility, we offer a SageMath worksheet~\parencite{code} for our computations.
By~\parencite[Theorem II.2]{mixon19}, there is a unique antipodal $2n$-packing (up to isometry) with coherence satisfying Lemma~\ref{mu8Lemma} and whose contact graph is isomorphic to $G$.
The singleton set containing this packing is a maximal connected set of optimal packings, thereby implying uniqueness.
We conclude in Section~\ref{sec.future work} with a discussion of possible directions for future work.

\section{Enumerating Contact Graph Candidates}\label{contactSection}

In this section, we discuss various combinatorial features of the contact graph of an antipodal $2n$-packing, and then we apply these features to enumerate contact graph candidates for the case where $n=8$.
First, for any $X\subseteq S^2$, notice that no two edges of $G(X)$ may cross, since otherwise, a simple application of the triangle inequality produces two points in $X$ that are closer than $\psi(X)$ apart.  This implies the following.

\begin{lem}
	For any $X \subseteq S^2$, $G(X)$ is a planar graph.
\end{lem}

We define the \emph{projective contact graph} $G([\Phi])$ of a projective $n$-packing $\Phi$ to have vertex set $\Phi$ and $x\leftrightarrow y$ precisely when $|\langle x,y\rangle| = \mu(\Phi)$.  The contact graph $G(X(\Phi))$ of the corresponding antipodal packing provides a canonical double cover of $G([\Phi])$, where the covering map amounts to identifying antipodal pairs.  This covering map naturally provides a \emph{projective planar} embedding of $G([\Phi])$.

\begin{lem}\label{projLemma}
	For any projective packing $\Phi$, $G([\Phi])$ is a projective planar graph.
\end{lem}

For finite $X \subseteq S^2$, we say $x \in X$ can be \emph{shifted} if every open neighborhood $U \subseteq S^2$ of $x$ contains a point $y \in U$ such that $d(x',y) > \psi(X)$ for all $x' \in X \setminus \{x\}$.  If the isolated vertices of $G(X)$ are the only ones that can be shifted, then we call $G(X)$ \emph{irreducible}.  Irreducible contact graphs were introduced in~\parencite{schutte51} in order to study the Tammes problem.  They made the following straightforward observations.

\begin{lem}\label{svLemma} Suppose $X \subseteq S^2$ with $|X| > 6$ and $G(X)$ irreducible.
	\begin{enumerate}
		\item[(i)] If $x \in X$ is not isolated in $G(X)$, then $3 \leq \deg(x) \leq 5$.
		\item[(ii)] Each face of $G(X)$ is a convex spherical polygon.
		\item[(iii)] The length of each face of $G(X)$ is at most $\lfloor 2\pi / \psi(X) \rfloor$.
	\end{enumerate}
\end{lem}

The proof of Lemma~\ref{svLemma}(iii) hinges on the following useful lemma.

\begin{lem}\label{lem.spherical polygons}
For every convex spherical $B\subseteq S^2$ and convex spherical polygon $A\subseteq B$, the perimeter of $A$ is at most the perimeter of $B$.
\end{lem}

\begin{proof}
There exist finitely many hemispheres $\{H_i\}$ such that $A=\bigcap_i H_i$.
Put $S_0:=B$ and $S_{k+1}:=S_k\cap H_{k+1}$ for each $k\geq 0$.
Then by the triangle inequality, the perimeter of $S_{k+1}$ is at most that of $S_k$ for every $k$, implying the result.
\end{proof}

\begin{proof}[Proof of Lemma~\ref{svLemma}]
If $\deg(x) = 1$, then $x$ can be shifted.  Similarly, if $\deg(x) = 2$ and $\psi(X) < \pi/2$, which holds for $|X| > 6$, then $x$ can be shifted.
If $\deg(x) > 5$, then the smallest angle between two edges emanating from $x$ is at most $\pi/3$.  This forces the corresponding vertices adjacent to $x$ to lie closer to one another than $\psi(X)$, which is not possible.  This establishes (i).  For (ii), if any face of $G(X)$ is not convex, then the non-convexity is witnessed by one of its vertices having interior angle greater than $\pi$. This vertex can then be shifted into the face, contradicting irreducibility.  Finally, (iii) follows from the fact that each face has perimeter at most $2\pi$, which in turn follows from Lemma~\ref{lem.spherical polygons} by taking $A$ to be the face in question, which is allowed by (ii), and $B$ to be any hemisphere that contains that face.
\end{proof}

Suppose $X \subseteq S$ contains a pair of antipodal points $x, -x \in X$.  As long as $|X| > 6$ so that $\psi(X) < \pi/2$, $x$ can be shifted if and only if $-x$ can be shifted as well.  It follows that by shifting antipodal vertices simultaneously, every antipodal $X \subseteq S^2$ with $|X| > 6$ and $G(X)$ reducible can be transformed to an antipodal $X' \subseteq S^2$ with $\psi(X') \geq \psi(X)$ such that either $G(X')$ is irreducible or $\psi(X') > \psi(X)$.  Hence, when classifying the contact graphs of optimal antipodal $2n$-packings $X$, it suffices to consider irreducible candidates for $G(X)$ and possible locations for isolated vertices.  In what follows, we make a few additional observations specific to antipodal sets, including a strengthening of \cref{svLemma}(iii).

\begin{lem}\label{antiLemma} Suppose $X \subseteq S^2$ is antipodal with $|X| > 6$ and $G(X)$ irreducible.
	\begin{enumerate}
		\item[(i)] The shortest path between an antipodal pair of vertices in $G(X)$ has at least $\lceil\pi/\psi(X)\rceil$ edges.
		\item[(ii)] Any two vertices incident to a common face of $G(X)$ are not antipodal.
		\item[(iii)]  Each face of $G(X)$ has length at most
	\[
		\left\lfloor \frac{2\pi}{\psi(X)} \sin\left(\frac{\pi - \psi(X)}{2}\right) \right\rfloor.
	\]
	\end{enumerate}
\end{lem}

\begin{proof}
	For (i), note that we need at least $\pi / \psi(X)$ edges of length $\psi(X)$ to form a path between $x, -x \in X$.  For (ii), suppose that $x$ and $-x$ were incident to a common face of $G(X)$.  This face must then consist of two disjoint paths of at least $\lceil{\pi / \psi(X)}\rceil$ edges.  Together with \cref{svLemma}(iii), we see that each path must form a geodesic of length $\pi$ and the face must form a spherical lune.  Since $\psi(X) < \pi/2$, we could then shift a point distinct from $x$, $-x$ into the face, demonstrating the reducibility of $G(X)$.  Hence, (ii) follows.  For (iii), observe that if a pair of points $x,y \in X$ are not an antipodal pair, then $d(x,y) \leq \pi - \psi(X)$; otherwise, the distance from $x$ to $-y$ would be less than $\psi(X)$.  Then every face of $G(X)$ lies in a spherical cap of geodesic diameter at most $\pi - \psi(X)$, which has perimeter at most $2\pi \sin(\frac{\pi - \psi(X)}{2})$.  \cref{lem.spherical polygons} then gives an upper bound on the perimeter of each convex face of $G(X)$, from which (iii) follows.
\end{proof}

A short argument involving \cref{lem.spherical polygons} forbids the isolated vertices of a contact graph from residing within a triangular face.  Additional restrictions for isolated vertices were observed in~\parencite[Lemma 2, Lemma 9]{boroczky03}.

\begin{lem}\label{isolatedLemma} Suppose $X \subseteq S^2$ and $|X| \geq 13$.
\begin{enumerate}
	\item[(i)] Every face of $G(X)$ containing an isolated vertex has length at least 6.
	\item[(ii)] No hexagonal face of $G(X)$ contains two isolated vertices.
\end{enumerate}
\end{lem}

\subsection{Generating Contact Graph Candidates}\label{gamma1Section}

We now construct a finite set of graphs which contains, up to isomorphism, the contact graph of an optimal antipodal 16-packing.  For the remainder of this section, let $\Phi$ denote an optimal projective 8-packing with corresponding optimal antipodal 16-packing $X = X(\Phi)$, where the discussion preceding \cref{svLemma} allows us to insist that $G(X)$ be irreducible without loss of generality.  By \cref{svLemma}(i), each non-isolated vertex in $X$ has degree 3, 4 or 5 in $G(X)$.  As $G(X)$ provides a double cover for $G([\Phi])$, each non-isolated vertex in $G([\Phi])$ has degree 3, 4 or 5 as well.

By deleting isolated vertices from a graph, we obtain its \emph{essential part}.  Observe that the essential part of $G([\Phi])$ is necessarily connected, since otherwise, $G(X)$ exhibits a face that is not simply connected, and therefore not a convex spherical polygon, violating \cref{svLemma}(ii). Moreover, the essential part of $G([\Phi])$ must have at least 6 vertices.  Otherwise, $G([\Phi])$ would have 3 isolated vertices, which would imply that the essential part of $G(X)$ consists of at most 5 antipodal pairs.  By \cref{isolatedLemma}(i), it would also contain a face of length 6, which is impossible since \cref{antiLemma}(ii) forbids antipodal pairs from lying on the same face.

We use \texttt{nauty}~\parencite{nauty} within SageMath~\parencite{sagemath} to generate connected graphs with order 6, 7, or 8, with minimum degree at least 3, and with maximum degree at most 5.  This produces a set $\Gamma_0$ of 1007 graphs, one of which must be isomorphic to the essential part of $G([\Phi])$.  To obtain candidates for $G(X)$, we generate all planar double covers for the graphs in $\Gamma_0$.  Given a graph $G$ on vertices $V$ with edges $E$, every double cover $\tilde{G}$ of $G$ can be represented on the vertex set $V \times \{0,1\}$ in the following way.  For each edge $\{v,w\} \in E$, either include in $\tilde{G}$ the edges $\{(v,0), (w,0)\}, \{(v,1), (w,1)\}$ or the edges $\{(v,0), (w,1)\}, \{(v,1), (w,0)\}$.  In this way, a graph with $m$ edges has $2^m$ double covers, many of which are often isomorphic.

We need only retain double covers that are connected, planar, and consistent with the constraints of this section.  It will be helpful to set $\psi_8 := \arccos(\mu_8)$; observe that $\psi(X) = \psi_8$ since $X$ is optimal by assumption.  \cref{mu8Lemma} provides the bounds
\begin{equation}\label{psi8Bound}
0.86638 \leq \psi_8 \leq 0.9273.
\end{equation}
The upper bound together with \cref{antiLemma}(i) ensures that every path between antipodal vertices $(v,0)$ and $(v,1)$ must travel along at least 4 edges; we eliminate graphs that violate this condition.  We also eliminate any graphs that contain a \emph{wheel} subgraph, or a cycle whose vertices all share a common neighbor.  Indeed, since the edge lengths of $G(X)$ are all equal, it could only contain a wheel graph on $n + 1$ vertices if the interior angle of each triangular face was $2\pi / n$; this is only possible on $S^2$ for $n = 3, 4, 5$, but each of these would require an edge length of $\psi(X) \geq 1.107$, contradicting \eqref{psi8Bound}.  Retaining only one isomorphic copy of each double cover that meets these requirements, we obtain a set $\Gamma_1$ of 547 planar graphs that contains an isomorphic copy of the essential part of $G(X)$.

Since we are considering reasonably small graphs, \texttt{nauty} generates $\Gamma_0$ very quickly.  However, extracting $\Gamma_1$ from $\Gamma_0$ is rather slow due to repeated graph isomorphism queries, taking several hours.  \cref{projLemma} indicates that we could first trim $\Gamma_0$ by insisting that each member be projective planar.  In principle, one could implement this step with the help of an existing fast algorithm for detecting projective planarity~\parencite[\S 15.8]{kocay17}.  Unfortunately, no such algorithm is currently available in SageMath, and so our computer-assisted proof does not incorporate this potential speedup.

\section{Enumerating Combinatorial Graph Embeddings}
\label{sec.combinatorial}

At this point, we have a sizable collection of contact graph candidates, and later, we will eliminate most of these by geometric considerations.
Before we can do this, we must identify the different planar embeddings of our contact graph candidates up to homeomorphism.
Such equivalence classes can be encoded combinatorially by associating with each vertex a counterclockwise cyclic ordering of its neighbors.
Given a planar embedding of a graph modulo homeomorphism, we call this encoding the corresponding \emph{combinatorial embedding} of that graph.
In this section, we describe a practical method to enumerate all combinatorial embeddings of a planar graph, which will allow us to recover all combinatorial embeddings for the graphs of $\Gamma_1$.

We first summarize some standard language from graph theory.  A \emph{cut vertex} for $G$ is a vertex whose deletion disconnects $G$; we say $G$ is \emph{separable} if it has a cut vertex or \emph{biconnected} otherwise.  The \emph{blocks} of a separable graph $G$ are its maximal biconnected subgraphs.  A \emph{separation pair} for $G$ is a pair of vertices whose simultaneous deletion disconnects $G$, and we say $G$ is \emph{triconnected} if it has no separation pair.

It is well known~\parencite{whitney33} that the triconnected planar graphs are exactly those with a unique combinatorial embedding, up to reversing the cyclic ordering of each vertex simultaneously.  As observed in \parencite{maclane37}, it is ``natural to try to reduce a graph $G$ which is not triply connected to triply connected constituents.''  For our purposes, since the combinatorial embeddings of the triconnected constituents are unique modulo reflection, such a decomposition would enable one to enumerate all combinatorial embeddings of the original graph.  Similar decompositions are crucial to efficient graph algorithms such as linear-time planarity tests~\parencite{hopcroft74}.

We will encode the triconnected constituents of graphs with the SPQR-tree data structure.  This was introduced in~\parencite{dibattista89} and provides an efficient representation of the set of all combinatorial embeddings of a biconnected planar graph~\parencite{dibattista96}.  The main idea is to decompose a biconnected graph $G$ into edge-disjoint subgraphs $\{G_i\}$; the combinatorial embeddings of $G$ can be obtained from those of $\{G_i\}$ by considering their possible relative arrangements.  If any of $\{G_i\}$ is triconnected, we can report their two combinatorial embeddings.  Otherwise, we can iterate the procedure for each biconnected $G_i$ and for each block of each separable $G_i$.  An example of such a decomposition is depicted in \cref{SPQRexampleFigure}.
While the formal definition of SPQR-trees is quite technical, we found the presentation in~\parencite{gutwenger10} to be particularly readable.

\begin{figure}[t!]
\begin{center}
\begin{subfigure}{0.5\textwidth}
\begin{center}
\begin{tikzpicture}
\SetVertexMath
\GraphInit[vstyle=Normal]
\Vertex[x = 1, y = 0] {g};
\Vertex[x = 1/2, y = 1] {f};
\Vertex[x = 2, y = 0] {a};
\Vertex[x = 3/2, y = 1] {d};
\Vertex[x = 5/2, y = 1] {b};
\Vertex[x = 2, y = 2] {e};
\Vertex[x = 3/2, y = -1] {c};
\Edges(e,d,b,e,f,g);
\Edges(g,c);
\Edges(g,a,c);
\Edges(b,a);
\Edges[style={bend right = 50}](f,c);
\end{tikzpicture}
\hspace{2em}
\begin{tikzpicture}
\SetVertexMath
\GraphInit[vstyle=Normal]
\Vertex[x = 1, y = 0] {g};
\Vertex[x = 1/2, y = 1] {f};
\Vertex[x = 2, y = 0] {a};
\Vertex[x = 3, y = 2] {d};
\Vertex[x = 5/2, y = 1] {b};
\Vertex[x = 2, y = 2] {e};
\Vertex[x = 3/2, y = -1] {c};
\Edges(e,d,b,e,f,g);
\Edges(g,c);
\Edges(g,a,c);
\Edges(b,a);
\Edges[style={bend right = 50}](f,c);
\end{tikzpicture}

\vspace{4em}

\begin{tikzpicture}
\SetVertexMath
\GraphInit[vstyle=Normal]
\Vertex[x = -1, y = 0] {g};
\Vertex[x = -1/2, y = 1] {f};
\Vertex[x = -2, y = 0] {a};
\Vertex[x = -3/2, y = 1] {d};
\Vertex[x = -5/2, y = 1] {b};
\Vertex[x = -2, y = 2] {e};
\Vertex[x = -3/2, y = -1] {c};
\Edges(e,d,b,e,f,g);
\Edges(g,c);
\Edges(g,a,c);
\Edges(b,a);
\Edges[style={bend left = 50}](f,c);
\end{tikzpicture}
\hspace{2em}
\begin{tikzpicture}
\SetVertexMath
\GraphInit[vstyle=Normal]
\Vertex[x = -1, y = 0] {g};
\Vertex[x = -1/2, y = 1] {f};
\Vertex[x = -2, y = 0] {a};
\Vertex[x = -3, y = 2] {d};
\Vertex[x = -5/2, y = 1] {b};
\Vertex[x = -2, y = 2] {e};
\Vertex[x = -3/2, y = -1] {c};
\Edges(e,d,b,e,f,g);
\Edges(g,c);
\Edges(g,a,c);
\Edges(b,a);
\Edges[style={bend left = 50}](f,c);
\end{tikzpicture}
\end{center}

\end{subfigure}
\hfill
\begin{subfigure}{0.4\textwidth}
\begin{center}
\begin{tikzpicture}

\draw (1.75,0) -- (1.75,3.5) -- (0,7) -- (-1.75,3.5);

\node[draw=black,minimum size=8.5em,fill=white] at (0,7)
{
\begin{tikzpicture}
\SetVertexMath
\GraphInit[vstyle=Normal]
\Vertex[x = 1/2, y = 1] {f};
\Vertex[x = 2, y = 0] {a};
\Vertex[x = 5/2, y = 1] {b};
\Vertex[x = 2, y = 2] {e};
\Edges(e,f);
\Edges(a,b);
\draw[name path=af1,gray,opacity=.5] (f) to [bend right] (a);
  \draw[name path=af2,gray,opacity=.5] (f) to [bend left] (a);  
  \tikzfillbetween[on layer=bg,
	of=af1 and af2
	]{gray,opacity=.5};
\draw[name path=eb1,gray,opacity=.5] (e) to [bend right=20] (b);
  \draw[name path=eb2,gray,opacity=.5] (e) to [bend left=20] (b);  
  \tikzfillbetween[on layer=bg,
	of=eb1 and eb2
	]{gray,opacity=.5};
\end{tikzpicture}
};

\node[draw=black,minimum size=8.5em,fill=white] at (-1.75,3.5)
{
\begin{tikzpicture}
\SetVertexMath
\GraphInit[vstyle=Normal]
\Vertex[x = 1, y = 0] {g};
\Vertex[x = 1/2, y = 1] {f};
\Vertex[x = 2, y = 0] {a};
\Vertex[x = 3/2, y = -1] {c};
\Edges(f,g);
\Edges(g,c);
\Edges(g,a,c);
\Edges[style={bend right=50}](f,c);
\Edges[style={dashed}](f,a);
\end{tikzpicture}
};

\node[draw=black,minimum size=8.5em,fill=white] at (1.75,3.5) 
{
\begin{tikzpicture}
\SetVertexMath
\GraphInit[vstyle=Normal]
\Vertex[x = 5/2, y = 1] {b};
\Vertex[x = 2, y = 2] {e};
\Edge[style={bend right=50}](b)(e);
\draw[name path=eb1,gray,opacity=.5] (e) to [bend left=20] (b);
  \draw[name path=eb2,gray,opacity=.5] (e) to [bend right=20] (b);  
  \tikzfillbetween[on layer=bg,
	of=eb1 and eb2
	]{gray,opacity=.5}
\Edge[style={dashed, bend right=50}](e)(b);
\end{tikzpicture}
};

\node[draw=black,minimum size=8.5em,fill=white] at (1.75,0)
{
\begin{tikzpicture}
	\SetVertexMath
  \GraphInit[vstyle=Normal]
	\Vertex[x = 3/2, y = 1] {d};
	\Vertex[x = 5/2, y = 1] {b};
	\Vertex[x = 2, y = 2] {e};
	\Edge[style={dashed}](e)(b);
	\Edges(e,d,b);
\end{tikzpicture}
};

\node at (.6,2.35) {P};
\node at (-2.9,2.35) {R};
\node at (-1.15,5.85) {S};
\node at (.6,-1.15) {S};

\end{tikzpicture}
\end{center}
\end{subfigure}
\end{center}
\caption{
The four combinatorial embeddings of a graph and the interior nodes of its SPQR-tree.
Each node of the SPQR-tree is labeled by a skeleton graph.  For every thickened gray virtual edges in a skeleton graph, there is a corresponding child node representing subgraph of the original graph.
Dashed reference edges denote the ancestry of the SPQR-tree node, and solid black edges correspond to edges in the original graph.
We can enumerate the combinatorial embeddings of a graph from its SPQR-tree by reversing the orientation of R node graphs, which are necessarily triconnected, and permuting the edges of P node graphs.
}\label{SPQRexampleFigure}
\end{figure}
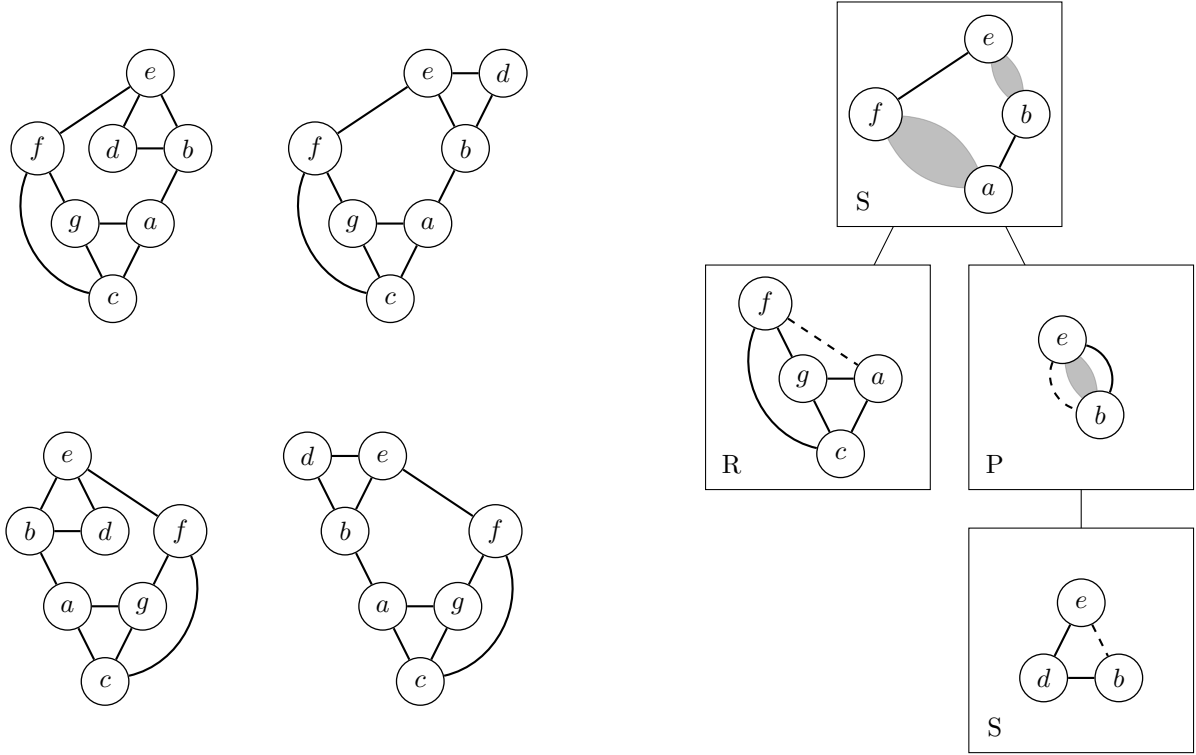

\subsection{Generating Combinatorial Embeddings}\label{e1Section}

In \cref{gamma1Section}, we generated a set $\Gamma_1$ consisting of 547 planar graphs, one of which is guaranteed to be the essential part of $G(X)$ for some optimal antipodal 16-packing $X$.  We now generate combinatorial embeddings of these graphs, and one of these necessarily corresponds to the planar embedding of $G(X)$ modulo homeomorphism and reflection.  Of the 547 graphs in $\Gamma_1$, 263 are triconnected, each producing a unique combinatorial embedding modulo reflection.  The remaining 284 graphs are biconnected, so we can produce reflection representatives of their combinatorial embeddings by constructing their SPQR-trees.  All together, we generate a set $\mathcal{E}_0$ of 2540 combinatorial embeddings, but only a fraction of these exhibit the necessary facial structure identified in the previous section.
In particular, by \cref{antiLemma}(iii), all faces must have length at most six, and by \cref{isolatedLemma}, there must be a hexagonal face for every isolated vertex.
The members of $\mathcal{E}_0$ meeting these requirements form a set $\mathcal{E}_1$ consisting of only 217 combinatorial embeddings.  Five of these have order 14, while the remaining 212 have order 16.  One of these must represent the combinatorial embedding of $G(X)$.

\section{Geometric Constraints}\label{geometricSection}

At this point, we have identified a set $\mathcal{E}_1$ comprised of 217 combinatorial embeddings of contact graph candidates, and at least one of these corresponds to an optimal antipodal 16-packing.
In this section, we leverage geometric constraints to eliminate almost all of these embeddings from consideration.
In principle, determining $\psi_8$ amounts to maximizing $\psi$ such that an embedding from $\mathcal{E}_1$ can be drawn on $S^2$ with equilateral edge length $\psi$.  To help describe the feasibility region of this optimization problem, we take the interior angles $\{\theta_i\}$ of the faces of a given embedding to be the decision variables.
These angles must satisfy certain relationships imposed by spherical geometry, which results in linear and nonlinear constraints.  We will relax each of these constraints to produce a linear program that must be feasible if the embedding corresponds to an optimal antipodal 16-packing.
Note that this same approach was taken in~\parencite{musin12, musin15} to solve instances of the Tammes problem.

A fundamental tool here is the \emph{spherical law of cosines}.

\begin{lem}\label{slocLemma}
If $a,b,c$ are geodesic side lengths for a spherical triangle with interior angle $\theta$ opposite $a$, then
\[
	\cos(a) = \cos(b)\cos(c) + \sin(b)\sin(c)\cos(\theta).	
\]
\end{lem}

Let $\alpha_n(d)$ denote the interior angle of a spherical regular $n$-gon of side length $d$.  We record some straightforward consequences of \cref{slocLemma}.

\begin{lem}\label{triangleLemma}
An equilateral spherical triangle with side length $\psi$ has interior angle
\[
	\alpha_3(\psi) = \arccos\left(\frac{\cos\left(\psi\right)}{1 + \cos\left(\psi\right)}\right).
\]
The opposite interior angles within a spherical rhombus are equal.  If a rhombus has side length $\psi$, then its adjacent interior angles $\theta, \theta'$ satisfy
\[
	\tan(\theta/2)\tan(\theta'/2) = \sec(\psi).
\]
In general, for an equilateral convex spherical polygon with $n$ sides of length $\psi$, we can express each of its interior angles $\theta_1, \ldots, \theta_n$ as a function of $\psi$ and any $n - 3$ of its interior angles.
\end{lem}

We have the following additional restrictions on the interior angles of a contact graph.

\begin{lem}\label{linearLemma}
Let $X \subseteq S^2$ and set $\psi = \psi(X)$.
\begin{enumerate}
	\item[(i)] Every interior angle $\theta$ of $G(X)$ satisfies $\alpha_3(\psi) \leq \theta \leq \pi$.
	\item[(ii)] The sum of the interior angles that meet at a vertex in $G(X)$ equals $2\pi$.
	\item[(iii)] The sum of the interior angles of a face of length $n$ in $G(X)$ is at most $n\alpha_n(\psi)$.
	\item[(iv)] Suppose $\alpha,\beta,\gamma,\delta,\epsilon$ are the interior angles of a pentagonal face of $G(X)$ ordered cyclically.  Then $\alpha \geq \epsilon$ only if $\delta \geq \beta$.
	\item[(v)] Suppose $\alpha,\beta,\gamma,\delta,\epsilon,\zeta$ are the interior angles of a hexagonal face of $G(X)$ ordered cyclically. Then $\alpha \geq \delta$ only if both $\gamma \geq \zeta$ and $\epsilon \geq \beta$.
%
\end{enumerate}
\end{lem}

\begin{proof}
For (i), note that every interior angle $\theta$ must be at least $\alpha(\psi)$ since vertices of $G(X)$ must be of geodesic distance $\psi$ apart, while the convexity of the faces guarantees $\theta \leq \pi$.  Next, (ii) is obvious.  Since the area of a spherical polygon is determined by the sum of its interior angles, (iii) follows from a version of the isoperimetric inequality that compares the face's area to that of a regular spherical $n$-gon of the same perimeter.  The proofs of (iv) and (v) can be found in \parencite[Proposition~5]{danzer86} and \parencite[Proposition~12]{boroczky03}, respectively.
\end{proof}

Antipodal packings satisfy additional constraints that do not appear in the relevant Tammes literature.

\begin{lem}\label{lem.path angles}
Let $X \subseteq S^2$ be antipodal and set $\psi = \psi(X)$.  Suppose $G(X)$ contains a path from $x\in X$ to $-x$ along four edges $e_1, e_2, e_3, e_4$.  Denote the angle between $e_1$ and $e_2$ by $\alpha$, the angle between $e_2$ and $e_3$ by $\beta$, and the angle between $e_3$ and $e_4$ by $\gamma$.  Then
\[
	\cos(\alpha) + \cos(\gamma) = -2\cot^2(\psi) \qquad \text{ and} \qquad \cos(\beta) \leq 1 - 2\cot^2(\psi).
	\]
\end{lem}

\begin{proof}
Suppose the path along $e_1, e_2, e_3, e_4$ passes through the vertices $x, x_1, x_2, x_3, -x$.  Connecting $x$ to $x_2$ by a geodesic segment results in a spherical triangle with side lengths $\psi, \psi, d(x,x_2)$, where $d(x,x_2)$ has opposite angle $\alpha$. Similarly, connecting $x_2$ to $-x$ by a geodesic segment results in a spherical triangle with side lengths $\psi, \psi, d(x_2, -x)$, where $d(x_2,-x)$ has opposite angle $\gamma$.  Moreover, we have $d(x,x_2) + d(x_2,-x) = \pi$ since $x$ and $-x$ are antipodal.  Applying \cref{slocLemma},
\[
	\cos(\alpha) + \cos(\gamma) = \frac{\cos(d(x,x_2)) - \cos^2(\psi)}{\sin^2(\psi)} + \frac{\cos(d(x_2,-x)) - \cos^2(\psi)}{\sin^2(\psi)} = -2\cot^2(\psi).
\]
Similarly, connecting $x_1$ to $x_3$ results in a spherical triangle with side lengths $\psi,\psi,d(x_1,x_3)$, where $d(x_1,x_3)$ has opposite angle $\beta$.  Moreover, we know that $d(x_1,x_3) \geq \pi - 2\psi$ since we are traveling between antipodal points $x$ and $-x$.  Applying \cref{slocLemma},
\[
	\cos(\beta) = \frac{\cos(d(x_1,x_3)) - \cos^2(\psi)}{\sin^2(\psi)} \leq \frac{-\cos(2\psi) - \cos^2(\psi)}{\sin^2(\psi)} = 1 - 2\cot^2(\psi). \qedhere
\]
\end{proof}

\subsection{Eliminating Combinatorial Embeddings}\label{eliminateSection}

From \cref{e1Section}, we have a set $\mathcal{E}_1$ of 217 combinatorial embeddings.  Fix an embedding $E$ in $\mathcal{E}_1$, let $\{\theta_i\}_{i\geq1}$ denote the interior angles of each of its faces, and set
\[
	m_i := \inf \theta_i \qquad \text{ and} \qquad	M_i := \sup \theta_i,
\]
where the infimum and supremum are taken over all drawings of $E$ on $S^2$ with equilateral geodesic edge length $\psi_8$.  Supposing that such a drawing exists, then for any index $i$, a conclusion of the form $m_i > M_i$ produces a contradiction, and so we may eliminate $E$ from consideration.

Recall that \cref{mu8Lemma} implies $0.86638 \leq \psi_8 \leq 0.9273$, and so \cref{triangleLemma} gives $1.1668 \leq \alpha_3(\psi_8) \leq 1.1864$.  Moreover, by \cref{linearLemma}(i), each interior angle $\theta_i$ satisfies $1.1668 \leq \theta_i \leq \pi$.  This provides us with initial lower bounds on each $m_i$ and initial upper bounds on each $M_i$.  Next, we relax by treating the edge length $\psi$ as another decision variable.
Denote
\[
m_0 := \inf \psi \qquad \text{and} \qquad M_0 := \sup \psi,
\]
where the infimum and supremum are taken over all drawings of the embedding with edge length $\psi$ satisfying $0.86638\leq \psi\leq0.9273$ and with interior angles $\theta_i\in[m_i, M_i]$.  Of course, if the embedding corresponds to an optimal contact graph, then $M_0 = \psi_8$.  We can use these bounds and simple interval arithmetic to relax each of the nonlinear constraints from the previous section to linear constraints and use linear programming to determine improved bounds on $m_i$ and $M_i$.
For instance, suppose we currently have bounds of the form $m_0\geq a$ and $M_0\leq b$. Then \cref{triangleLemma} and \cref{linearLemma}(i) together imply
\[
m_i
\geq\alpha_3(\psi)
\geq\arccos\bigg(\frac{\cos(a)}{1+\cos(b)}\bigg)
\]
for every $i\geq1$.
In a similar way, the remainder of \cref{triangleLemma} and also \cref{lem.path angles} can be used to improve estimates of various $m_i$ and $M_i$ from existing estimates of other $m_j$ and $M_j$.
In addition, every part of \cref{linearLemma} converts existing bounds on $m_0$ and $M_0$ into inequality constraints on all of $\{\theta_i\}_{i\geq1}$, allowing one to improve existing estimates on each $m_i$ and $M_i$ by linear programming.
As one might expect, this process benefits from iteration.  If our estimates of any $m_i$ and $M_i$ ever imply $m_i>M_i$, then we may eliminate the corresponding embedding from consideration.  For each member of $\mathcal{E}_1$, we follow this general procedure for several iterations, and as a result, only 17 combinatorial embeddings remain.

For the remaining 17 embeddings, we improve our bounds on each $m_i$ and $M_i$ as follows.  Suppose that we know $m_0 \geq a$ and $M_0 \leq b$.  Then we can partition $[a,b]$ into several subintervals $[a,b] = [a_1, b_1] \cup \cdots \cup [a_k, b_k]$ and repeat the same procedure as in the previous paragraph under the assumption $\psi \in [a_j,b_j]$.  Each linear program involves relaxations of several nonlinear constraints on the interior angles $\theta_i$ involving $\psi$, so it comes as no surprise that we have stronger bounds on each $\theta_i$ when we restrict $\psi$ to a small subinterval.  For several embeddings, each subinterval $[a_j,b_j]$ results in an infeasible program, allowing us to eliminate that embedding.  In other cases, we instead obtain improved bounds on each $m_i$ and $M_i$.  After this stage, only seven combinatorial embeddings remain.

We remove a few more embeddings from consideration by simultaneously bisecting feasiblility regions for $\psi$ and each interior angle for a given face.  For a face of length $n$, this results in $2^{n + 1}$ regions that can be checked for feasibility individually.  Two of our embeddings lead to infeasible programs in all $2^7$ regions corresponding to one of their hexagonal faces.  For others, this again leads to improved upper bounds on $\psi$, and iterating our procedure finally shows that one last embedding leads to an infeasible program.

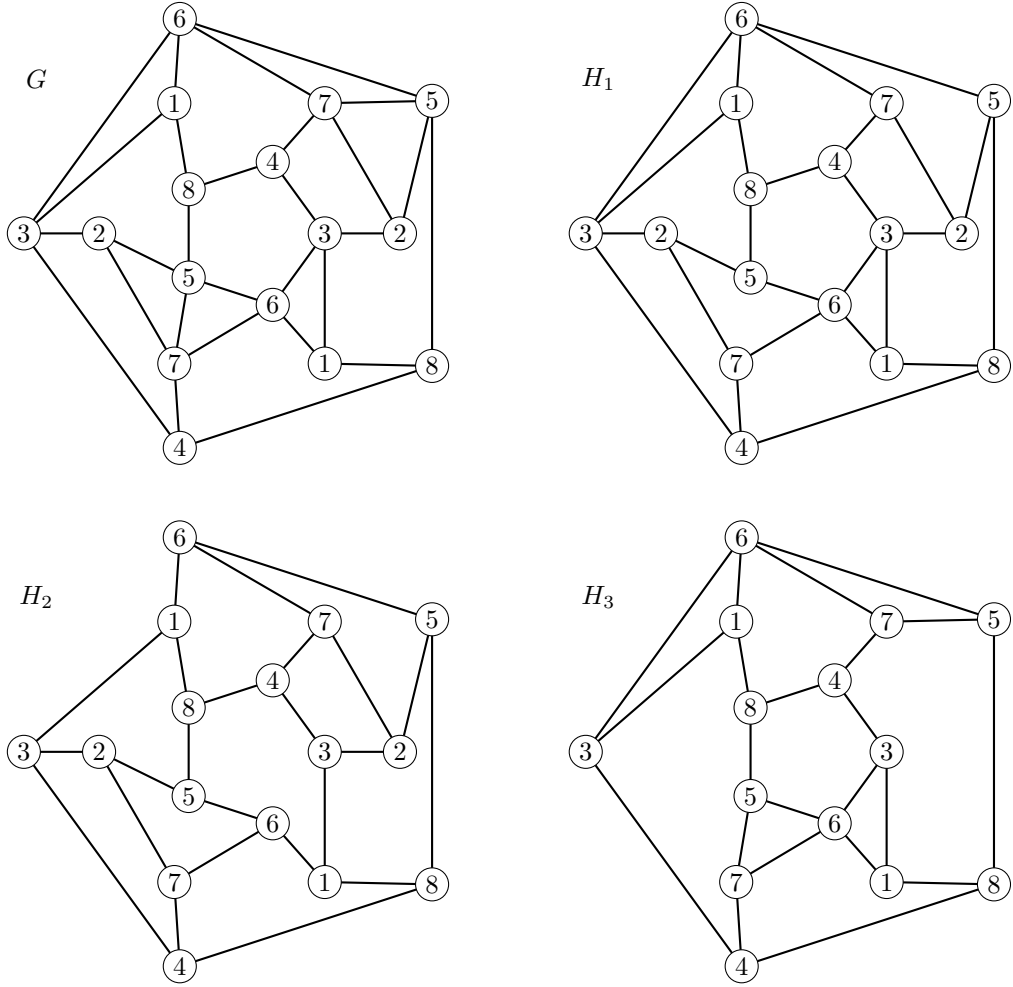
\begin{figure}
\hfill
\begin{tikzpicture}

\node at (144:3.5) {$G$};

\SetVertexMath
\GraphInit[vstyle=Normal]

\tikzset{VertexStyle/.style = {
inner sep = 0pt,
outer sep = 0pt,
minimum size = 1.25em,
shape = circle,
draw}}

\Vertex[a = 0,d = 1,L={3}]{x_3}
\Vertex[a = 72,d = 1,L={4}]{x_4}
\Vertex[a = 144,d = 1,L={8}]{x_8}
\Vertex[a = 216,d = 1,L={5}]{x_5}
\Vertex[a = 288,d = 1,L={6}]{x_6}

\Vertex[a=0,d=2,L={2}]{-x_2}
\Vertex[a=60,d=2,L={7}]{-x_7}
\Vertex[a=120,d=2,L={1}]{-x_1}
\Vertex[a=180,d=2,L={2}]{x_2}
\Vertex[a=240,d=2,L={7}]{x_7}
\Vertex[a=300,d=2,L={1}]{x_1}

\Vertex[a = 36,d = 3,L={5}]{-x_5}
\Vertex[a = 108,d = 3,L={6}]{-x_6}
\Vertex[a = 180,d = 3,L={3}]{-x_3}
\Vertex[a = 252,d = 3,L={4}]{-x_4}
\Vertex[a = 324,d = 3,L={8}]{-x_8}

\Edges(x_3,x_4,x_8,x_5,x_6,x_3);
\Edges(-x_3,-x_4,-x_8,-x_5,-x_6,-x_3);
\Edges(x_5,x_2,-x_3,-x_1,x_8);
\Edges(-x_1,-x_6,-x_7,x_4);
\Edges(-x_5,-x_2,-x_7);
\Edges(-x_2,x_3,x_1,x_6,x_7);
\Edges(x_1,-x_8);
\Edges(-x_4,x_7,x_2);
\Edges(-x_5,-x_7);
\Edges(x_5,x_7);

\end{tikzpicture}
\hfill
\begin{tikzpicture}

\node at (144:3.5) {$H_1$};

\SetVertexMath
\GraphInit[vstyle=Normal]

\tikzset{VertexStyle/.style = {
inner sep = 0pt,
outer sep = 0pt,
minimum size = 1.25em,
shape = circle,
draw}}

\Vertex[a = 0,d = 1,L={3}]{x_3}
\Vertex[a = 72,d = 1,L={4}]{x_4}
\Vertex[a = 144,d = 1,L={8}]{x_8}
\Vertex[a = 216,d = 1,L={5}]{x_5}
\Vertex[a = 288,d = 1,L={6}]{x_6}

\Vertex[a=0,d=2,L={2}]{-x_2}
\Vertex[a=60,d=2,L={7}]{-x_7}
\Vertex[a=120,d=2,L={1}]{-x_1}
\Vertex[a=180,d=2,L={2}]{x_2}
\Vertex[a=240,d=2,L={7}]{x_7}
\Vertex[a=300,d=2,L={1}]{x_1}

\Vertex[a = 36,d = 3,L={5}]{-x_5}
\Vertex[a = 108,d = 3,L={6}]{-x_6}
\Vertex[a = 180,d = 3,L={3}]{-x_3}
\Vertex[a = 252,d = 3,L={4}]{-x_4}
\Vertex[a = 324,d = 3,L={8}]{-x_8}

\Edges(x_3,x_4,x_8,x_5,x_6,x_3);
\Edges(-x_3,-x_4,-x_8,-x_5,-x_6,-x_3);
\Edges(x_5,x_2,-x_3,-x_1,x_8);
\Edges(-x_1,-x_6,-x_7,x_4);
\Edges(-x_5,-x_2,-x_7);
\Edges(-x_2,x_3,x_1,x_6,x_7);
\Edges(x_1,-x_8);
\Edges(-x_4,x_7,x_2);

\end{tikzpicture}
\hfill \phantom{.}

\vspace{2em}

\hfill
\begin{tikzpicture}
\node at (144:3.5) {$H_2$};
\SetVertexMath
\GraphInit[vstyle=Normal]

\tikzset{VertexStyle/.style = {
inner sep = 0pt,
outer sep = 0pt,
minimum size = 1.25em,
shape = circle,
draw}}

\Vertex[a = 0,d = 1,L={3}]{x_3}
\Vertex[a = 72,d = 1,L={4}]{x_4}
\Vertex[a = 144,d = 1,L={8}]{x_8}
\Vertex[a = 216,d = 1,L={5}]{x_5}
\Vertex[a = 288,d = 1,L={6}]{x_6}

\Vertex[a=0,d=2,L={2}]{-x_2}
\Vertex[a=60,d=2,L={7}]{-x_7}
\Vertex[a=120,d=2,L={1}]{-x_1}
\Vertex[a=180,d=2,L={2}]{x_2}
\Vertex[a=240,d=2,L={7}]{x_7}
\Vertex[a=300,d=2,L={1}]{x_1}

\Vertex[a = 36,d = 3,L={5}]{-x_5}
\Vertex[a = 108,d = 3,L={6}]{-x_6}
\Vertex[a = 180,d = 3,L={3}]{-x_3}
\Vertex[a = 252,d = 3,L={4}]{-x_4}
\Vertex[a = 324,d = 3,L={8}]{-x_8}

\Edges(x_3,x_4,x_8,x_5,x_6);
\Edges(-x_3,-x_4,-x_8,-x_5,-x_6);
\Edges(x_5,x_2,-x_3,-x_1,x_8);
\Edges(-x_1,-x_6,-x_7,x_4);
\Edges(-x_5,-x_2,-x_7);
\Edges(-x_2,x_3,x_1,x_6,x_7);
\Edges(x_1,-x_8);
\Edges(-x_4,x_7,x_2);

\end{tikzpicture}
\hfill
\begin{tikzpicture}
\node at (144:3.5) {$H_3$};
\SetVertexMath
\GraphInit[vstyle=Normal]

\tikzset{VertexStyle/.style = {
inner sep = 0pt,
outer sep = 0pt,
minimum size = 1.25em,
shape = circle,
draw}}

\Vertex[a = 0,d = 1,L={3}]{x_3}
\Vertex[a = 72,d = 1,L={4}]{x_4}
\Vertex[a = 144,d = 1,L={8}]{x_8}
\Vertex[a = 216,d = 1,L={5}]{x_5}
\Vertex[a = 288,d = 1,L={6}]{x_6}

\Vertex[a=60,d=2,L={7}]{-x_7}
\Vertex[a=120,d=2,L={1}]{-x_1}
\Vertex[a=240,d=2,L={7}]{x_7}
\Vertex[a=300,d=2,L={1}]{x_1}

\Vertex[a = 36,d = 3,L={5}]{-x_5}
\Vertex[a = 108,d = 3,L={6}]{-x_6}
\Vertex[a = 180,d = 3,L={3}]{-x_3}
\Vertex[a = 252,d = 3,L={4}]{-x_4}
\Vertex[a = 324,d = 3,L={8}]{-x_8}

\Edges(x_3,x_4,x_8,x_5,x_6,x_3);
\Edges(-x_3,-x_4,-x_8,-x_5,-x_6,-x_3);
\Edges(-x_3,-x_1,x_8);
\Edges(-x_1,-x_6,-x_7,x_4);
\Edges(x_3,x_1,x_6,x_7);
\Edges(x_1,-x_8);
\Edges(-x_4,x_7);
\Edges(-x_5,-x_7);
\Edges(x_5,x_7);

\end{tikzpicture}
\hfill \phantom{.}

\caption{The four combinatorial embeddings of contact graph candidates remaining after \cref{geometricSection}.  Antipodal vertices share the same vertex label.  $H_1$ is obtained from $G$ by deleting two antipodal edges, $H_2$ is obtained from $H_1$ by deleting two antipodal edges, and $H_3$ is obtained from $G$ by deleting a pair of antipodal vertices.\label{fig.sloane graphs}}

\end{figure}

At this point, we have reduced our search to four combinatorial embeddings, which we denote by $G$, $H_1$, $H_2$, and $H_3$; see \cref{fig.sloane graphs} for an illustration.
One may verify that $G$ corresponds to the numerical packing obtained by Conway, Hardin and Sloane~\parencite{conway96}, which we made exact in~\parencite{mixon19}.
Each $H_\ell$ can be obtained from $G$ by deleting appropriate edges or vertices.
Since deleting an edge or isolating a vertex from a contact graph may correspond to an infinitesimal shift of its vertices, we should not expect the coarse approximations in the linear program described so far to be precise enough to rule out any of $H_1$, $H_2$, or $H_3$.
In the following section, we close this gap by passing to a dual optimization program.

\section{Optimality Conditions and Uniqueness}
\label{sec.opt and uniqueness}

To complete the proof of the main result, we first apply optimization theory to analyze each of the remaining four contact graph candidates (eliminating all but one), and then we leverage ideas from real algebraic geometry to demonstrate uniqueness.
For the first step, we follow~\parencite{musin15} by testing whether each candidate is consistent with the necessary Karush--Kuhn--Tucker optimality conditions~\parencite[\S 5.5.3]{boyd04}.
In words, these conditions imply that optimal packings are infinitesimally rigid and therefore must have an associated equilibrium stress matrix, which we denote by $(\omega_{ij})$.

\begin{lem}\label{stressLemma}
	Let $X = \{x_1, \ldots, x_{2n}\}$ be an optimal antipodal $2n$-packing.  Set
	\[
	J(i) := \{j : x_i \leftrightarrow x_j \text{ in } G(X)\},
	\]
and for each $j \in J(i)$, let $v_{ij}$ denote the unit vector at $x_i$ tangent to the geodesic from $x_i$ to $x_j$. There exist coefficients $\omega_{ij} = \omega_{ji} \in \mathbf{R}$ for $1 \leq i, j \leq {2n}$ with $i\neq j$ such that
\begin{enumerate}
\item[(i)] $\omega_{ij} \geq 0$ for all $i,j$ with $i\neq j$,
\item[(ii)] $\omega_{ij} = 0$ whenever $x_i$ and $x_j$ are not adjacent in $G(X)$,
\item[(iii)] $\sum_{i,j:i\neq j} \omega_{ij} = 1$, and
\item[(iv)] for each $i$, $\sum_{j \in J(i)} \omega_{ij} v_{ij} = 0$.
\end{enumerate}
\end{lem}

\begin{proof}
For any antipodal $2n$-packing $Z=\{z_1,\ldots,z_{2n}\}$, we may take $z_{j + n} = -z_j$ for $1 \leq j \leq n$ without loss of generality by reordering if necessary.
As such, the antipodal $2n$-packing problem in $S^2$ is equivalent to the following program, whose decision variables are $z_1,\ldots,z_{2n}\in\mathbf{R}^3$ and $\delta\in\mathbf{R}$:
\[
\begin{array}{rl}
\text{maximize} & \delta\\
\text{subject to} & \delta - |z_i - z_j|^2 \leq 0 \text{ for all } i,j \text{ with } i\neq j,\\
& |z_i|^2 - 1 = 0\text{ for all } i, \text{ and}\\
& z_{i + n} + z_i = 0 \text{ for all } 1 \leq i \leq n.
\end{array}
\]
For the optimal packing $X$, we denote the square of its minimum Euclidean distance by $\Delta:=2 - 2\cos(\psi(X))$.
Then $\Delta$ is the maximum value of the above program with maximizer $(x_1,\ldots,x_{2n},\Delta)$.
For what follows, it is convenient to put $z:=(z_1,\ldots,z_{2n})$ and define
	\begin{align*}
	f(z,\delta) &= \delta,\\
	g_{ij}(z,\delta) &= \delta - |z_i - z_j|^2  \text{ for } i,j \text{ with } i\neq j,\\
	h_i(z,\delta) &= |z_i|^2 - 1 \text{ for } 1 \leq i \leq 2n, \text{ and}\\
	h'_i(z,\delta) &= z_{i + n} + z_i \text{ for } 1 \leq i \leq n.
	\end{align*}
	The Karush--Kuhn--Tucker conditions guarantee coefficients $\omega_{ij}$, $\lambda_i$, and $\lambda_i'$ with
	\begin{equation}\label{lagrange}
	\nabla f (x,\Delta) = \sum_{i=1}^{2n}\sum_{\substack{j=1\\j\neq i}}^{2n} \omega_{ij} \nabla g_{ij}(x,\Delta) + \sum_{i = 1}^{2n} \lambda_i \nabla h_i(x,\Delta) + \sum_{i = 1}^n \lambda_i' \nabla h_i'(x,\Delta).
	\end{equation}
By symmetry, we may replace both $\omega_{ij}$ and $\omega_{ji}$ with their average, thereby ensuring $\omega_{ij}=\omega_{ji}$.
Dual feasibility ensures $\omega_{ij} \geq 0$ for all $i,j$ with $i\neq j$, implying (i).   Complementary slackness ensures $\omega_{ij} g_{ij}(x,\Delta) = 0$. As such, when $j \not \in J(i)$ and $j\neq i$, then $g_{ij}(x,\Delta) \neq 0$, and so $\omega_{ij} = 0$.  This gives (ii). Notice that (iii) follows by considering the $\delta$-component of the gradient vector in~\eqref{lagrange}, while its $z_i$-component yields 
	\[
	0 = \sum_{j \in J(i)} 2\omega_{ij}(x_j - x_i) + 2(\lambda_i + \lambda_i') x_i.
	\]
To obtain (iv), project the above identity onto the $2$-dimensional subspace that is parallel to the tangent plane to $S^2$ at $x_i$, and then observe that for $j\in J(i)$, the vectors $x_j-x_i$ have the same norm, and therefore the same component in the $x_i$ direction since $x_i,x_j\in S^2$.
\end{proof}

\subsection{Eliminating Subgraphs}  \cref{stressLemma} reports that every optimal antipodal packing $X$ must have a corresponding stress matrix $(\omega_{ij})$.
Note that we can verify whether a given $(\omega_{ij})$ is a stress matrix once we have the graph $G(X)$ and the tension directions $(v_{ij})$, both of which are determined by $X$.
In this subsection, we argue that for each combinatorial embedding $H_\ell$ illustrated in \cref{fig.sloane graphs}, then for every choice of tension directions $(v_{ij})$ that are consistent with the bounds derived in the previous section, there is no stress matrix $(\omega_{ij})$ that is consistent with both $H_\ell$ and $(v_{ij})$.

Given $H_\ell$, suppose for the sake of contradiction that there exists an optimal antipodal $16$-packing $X$ such that the essential part of $G(X)$ belongs to the homeomorphism equivalence class of $H_\ell$.
We start by putting \cref{stressLemma}(iv) in a standard form.
For each $i$, we apply an isometry $Q_i$ that maps the tangent plane of $S^2$ at $x_i$ to $\mathbf{R}^2$ in such a way that one of the tension vectors $v_{ij}$ is mapped to $(1,0)$.
Put $u_{ij}:=Q_i(v_{ij})$.
Then \cref{stressLemma}(iv) is equivalent to $\sum_{j\in J(i)}\omega_{ij}u_{ij}=0$.
Importantly, each $u_{ij}$ takes the form $u_{ij}=(\cos(\theta_{ij}),\sin(\theta_{ij}))$, where each $\theta_{ij}$ is a sum of interior angles $\theta_{ij}=\sum_{k\in K_{ij}}\theta_k$ over an index set $K_{ij}$ that is determined by the combinatorial embedding of $H_\ell$.
For each $k$, the previous section computed bounds on $m_k$ and $M_k$ such that $\theta_k\in[m_k,M_k]$, and as we will see, we can apply the relationships in \cref{stressLemma} to further improve these bounds.

%

Notice that bounds on the interior angles $\{\theta_k\}$ imply bounds on the angles $\theta_{ij}=\sum_{k\in K_{ij}}\theta_k$, which in turn imply bounds on the coordinates of $u_{ij}$ that take the form $\cos(\theta_{ij}) \in [c_{ij},C_{ij}]$ and $\sin(\theta_{ij}) \in [s_{ij},S_{ij}]$.  If $(\omega_{ij})$ satisfies \cref{stressLemma}(iv), then it must be the case that
\begin{align}
\sum_{j \in J(i)} \omega_{ij} c_{ij} &\leq 0 \leq \sum_{j \in J(i)} \omega_{ij} C_{ij},\text{ and} \nonumber\\
\sum_{j \in J(i)} \omega_{ij} s_{ij} &\leq 0 \leq \sum_{j \in J(i)} \omega_{ij} S_{ij}.\label{sineCosine}
\end{align}
By linear programming, we can quickly determine whether the set $\Omega$ consisting of all $(\omega_{ij})$ that satisfy both \cref{stressLemma}(i)--(iii) and \eqref{sineCosine} is empty.
Overall, we can sharpen the bounds from \cref{eliminateSection} by imposing the additional constraint that $\Omega$ is nonempty.
With this approach, we repeatedly partition the admissible intervals for each variable from \cref{eliminateSection}, improving the bounds on each $m_k$ and $M_k$ until we demonstrate $m_k>M_k$ for some $k$, thereby delivering the desired contradiction.

\subsection{The Optimal Configuration}

By the process of elimination, every maximal connected set of optimal antipodal $16$-packings in $S^2$ has a representative packing $X$ whose contact graph belongs to the homeomorphism equivalence class of either $G$ or its reflection.
Furthermore, no other irreducible graph is the contact graph of an optimal antipodal $16$-packing.
Since $G$ is irreducible with no isolated vertices, we may conclude that the contact graph of every optimal antipodal $16$-packing is isomorphic to $G$.

By selecting one vector from each antipodal pair in $X$, we obtain an optimal projective $8$-packing $\Phi = \{\varphi_1, \ldots, \varphi_8\}$.  If we abuse notation by treating $\Phi$ as a matrix with columns $\varphi_1, \ldots, \varphi_8$, then its \emph{Gram matrix} is given by $\Phi^T\Phi$.  Notice that $(\Phi^T\Phi)_{ij} = \mu_8$ precisely when $\varphi_i$ and $\varphi_j$ are adjacent in $G(X)$, while $(\Phi^T\Phi)_{ij} = -\mu_8$ precisely when $\varphi_i$ and $-\varphi_j$ are adjacent in $G(X)$.  Every other entry of $\Phi^T\Phi$ is strictly less than $\mu_8$ in absolute value.
We choose $\varphi_1, \ldots, \varphi_8\in X$ so that the resulting Gram matrix takes the form
\[
\Phi^T\Phi = \left(\begin{array}{rrrrrrrr}
1 & a_1 & \cellcolor{lightgray}{\mu_8} & a_2 & a_3 & \cellcolor{lightgray}{-\mu_8} & a_4 & \cellcolor{lightgray}{-\mu_8} \\
a_1 & 1 & \cellcolor{lightgray}{-\mu_8} & a_5 & \cellcolor{lightgray}{-\mu_8} & a_6 & \cellcolor{lightgray}{-\mu_8} & a_7 \\
\cellcolor{lightgray}{\mu_8} & \cellcolor{lightgray}{-\mu_8} & 1 & \cellcolor{lightgray}{\phantom{-}\mu_8} & a_8 & \cellcolor{lightgray}{-\mu_8} & a_9 & a_{10} \\
a_2 & a_5 & \cellcolor{lightgray}{\mu_8} & 1 & a_{11} & a_{12} & \cellcolor{lightgray}{\mu_8} & \cellcolor{lightgray}{\mu_8} \\
a_3 & \cellcolor{lightgray}{-\mu_8} & a_8 & a_{11} & 1 & \cellcolor{lightgray}{\mu_8} & \cellcolor{lightgray}{\mu_8} & \cellcolor{lightgray}{-\mu_8} \\
\cellcolor{lightgray}{-\mu_8} & a_6 & \cellcolor{lightgray}{-\mu_8} & a_{12} & \cellcolor{lightgray}{\mu_8} & 1 & \cellcolor{lightgray}{\mu_8} & a_{13} \\
a_4 & \cellcolor{lightgray}{-\mu_8} & a_9 & \cellcolor{lightgray}{\mu_8} & \cellcolor{lightgray}{\mu_8} & \cellcolor{lightgray}{\mu_8} & 1 & a_{14} \\
\cellcolor{lightgray}{-\mu_8} & a_7 & a_{10} & \cellcolor{lightgray}{\mu_8} & \cellcolor{lightgray}{-\mu_8} & a_{13} & a_{14} & 1
\end{array}\right).
\]
Overall, $\Phi^T\Phi$ is rank-$3$ and positive semidefinite, $a_j$ satisfies $|a_j| < \mu_8$ for every $1 \leq j \leq 14$, and $\mu_8$ satisfies \cref{mu8Lemma}.
In~\parencite{mixon19}, we used cylindrical algebraic decomposition to establish that there is a unique real matrix of this form.  Moreover, we computed $\mu_8$ exactly as the third-largest root of
\[
	1 + 5x - 8x^2 - 80x^3 - 78x^4 + 146x^5 - 80x^6 - 584x^7 + 677x^8 + 1537x^9,
\]
which is given numerically by $\mu_8 \approx 0.6475889787$.
This completes the proof of the main result.

\section{Future Work}
\label{sec.future work}

In this paper, we identified an optimal projective $8$-packing in $\mathbf{RP}^2$, and we established that it is unique up to isometry.
The unique optimality of this packing was predicted in~\parencite{conway96}, where numerical solutions for putatively optimal projective $n$-packings in $\mathbf{RP}^2$ were presented for $n \leq 55$.
For the $n=9$ case, they conjecture that there are two distinct isometry classes of optimal projective packings.
It would be natural to extend our work to study these larger packings, and such extensions would benefit from a more computationally efficient procedure.
The primary bottleneck appears in \cref{contactSection}, where, given a projective contact graph candidate in $\Gamma_0$ on $m$ edges, we naively produced all of its $2^m$ double covers, and then we retained a single isomorphic copy of each double cover that happened to be planar and consistent with other combinatorial constraints.
As indicated by \cref{projLemma}, we only need to consider members of $\Gamma_0$ that are projective planar, but we did not leverage this information in our computation; as such, we suspect that a fast projective planarity algorithm~\parencite[\S 15.8]{kocay17} could speed up our procedure.
Also, does there exist a faster-than-naive algorithm for generating planar double covers of a given projective planar graph?
Finally, it would be interesting to investigate comparable approaches for efficiently packing points in $\mathbf{FP}^{d-1}$ for $\mathbf{F} \in \{\mathbf{R},\mathbf{C}\}$ with $d \geq 3$ and, more generally, packing higher-dimensional subspaces in $\mathbf{F}^d$.

\section*{Acknowledgments}

DGM was partially supported by AFOSR FA9550-18-1-0107, NSF DMS 1829955, and the Simons Institute of the Theory of Computing.

\printbibliography

\end{document}